\documentclass[11pt]{article}
\usepackage{amsmath, amsthm,amsfonts,amssymb,amscd}
\usepackage{color}\usepackage { xcolor }
\usepackage{cancel}
\usepackage[normalem]{ulem}
\usepackage{geometry}
\usepackage{tikz}
\usepackage{subfigure}
\numberwithin{equation}{section}
\usepackage[T1]{fontenc}
\usepackage{cite}
\usepackage{enumerate}

\usepackage{float}
\usepackage{soul}
\usepackage{graphicx}
\usepackage{placeins}
\usepackage{float}
\usepackage{marginnote}
\newcommand{\Limsup}{\mathop{{\rm Lim}\,{\rm sup}}}

\def\tto{\;{\lower 1pt \hbox{$\rightarrow$}}\kern -10pt
	\hbox{\raise 2pt \hbox{$\rightarrow$}}\;}

\def\ra{\rangle}
\def\la{\langle}

\def\B{\mathbb B}

\def\h{\hfill\Box}
\def\R{\mathbb R}

\def\ox{\bar{x}}

\def\h{\hfill\triangle}

\newcounter{lk}
\def\Limsup{\mathop{{\rm Lim}\,{\rm sup}}}

\begin{document}
	
	\newtheorem{Theorem}{Theorem}[section]
	\newtheorem{Proposition}[Theorem]{Proposition}
	\newtheorem{Remark}[Theorem]{Remark}
	\newtheorem{Lemma}[Theorem]{Lemma}
	\newtheorem{Corollary}[Theorem]{Corollary}
	\newtheorem{Definition}[Theorem]{Definition}
	\newtheorem{Example}[Theorem]{Example}
	\newtheorem{Fact}[Theorem]{Fact}
	\renewcommand{\theequation}{\thesection.\arabic{equation}}
	\normalsize
	\def\proof{
		\normalfont
		\medskip
		{\noindent\itshape Proof.\hspace*{6pt}\ignorespaces}}
	\def\endproof{$\h$ \vspace*{0.1in}}

	\title{\small \bf VARIATIONAL ANALYSIS OF METRIC PROJECTIONS ONTO ISOTONE PROJECTION CONES VIA CODERIVATIVES}
	\date{}
	
	\author{Le Van Hien\footnote{Faculty of Education, HaTinh University, Hatinh, Vietnam; email: hien.levan@htu.edu.vn.}}
	
	\maketitle
	
	{\small \begin{abstract}
		In this paper, we study variational properties of the metric projection mapping onto isotone projection cones in finite-dimensional Euclidean spaces. We derive explicit formulas for both the Fréchet coderivative and the Mordukhovich coderivative of the projection operator. The analysis is based on a local description of the projection mapping via an associated generating system in a neighborhood of a reference point, which leads to computable coderivative characterizations.			
		As an application, we compute the covering constant of the projection mapping, providing a quantitative description of its local regularity. Furthermore, we establish verifiable sufficient conditions for the Aubin property of the solution mapping associated with parametric nonlinear complementarity problems associated with isotone projection cones.
		
		The obtained results contribute to the variational analysis of metric projections in a general cone setting where orthogonality arguments are not available, and to the stability theory of complementarity systems in finite dimensions.

	\end{abstract}}
	
	\medskip
	
	\normalsize
	\section{Introduction}
	\setcounter{equation}{0}
Projection operators onto convex sets play a fundamental role in optimization, variational inequalities, and equilibrium problems. In particular, projections onto structured cones arise naturally in complementarity systems and ordered models, where monotonicity properties are essential. Among these, isotone projection cones form an important class due to their intrinsic order structure and their relevance in applications; see \cite{IN90,IN92,N09}.

There have been numerous studies on projection operators onto special classes of sets in Hilbert spaces, including the derivation of explicit projection formulas, the computation of generalized derivatives, and various applications; see, for example, \cite{FP82,H77,H24,H24(2),HQ,Li24.2,OS08,R17,S87}. In particular, explicit formulas for projections onto lattice cones, together with their applications, and especially those concerning isotone projection cones, have been systematically developed by N\'emeth and co-authors in \cite{IN90,IN92,NN10}. These results provide valuable insight into the structure of projection mappings and serve as a foundation for further analytical developments.

Complementarity problems constitute important models in optimization, economics, physics, and engineering. In particular, complementarity problems associated with isotone projection cones have attracted considerable attention due to their wide range of applications. It is well known (see \cite{IN90}) that such problems can be equivalently reformulated as fixed-point problems involving projection-type mappings. This equivalence highlights the importance of studying the properties of projection operators onto cones, as well as the development of efficient methods for computing such projections.

In parallel, the study of stability and sensitivity of solution mappings has become a central topic in modern optimization. In this context, tools from variational analysis, particularly generalized differentiation, have proven to be highly effective \cite{DR14,HOS12,LM04,M18,M24,MO07,RW98,YZ17}. In particular, the Mordukhovich and Fréchet coderivatives provide a powerful framework for characterizing stability properties such as the Aubin property, metric regularity, and calmness. For example, Outrata and Sun \cite{OS08} employed coderivative techniques to analyze stability properties of complementarity problems associated with the second-order cone via projection onto this cone.

Motivated by the important role of projection operators onto isotone projection cones in complementarity problems and related applications, it is natural to investigate their variational properties in greater detail. Several studies have already addressed coderivative formulas for projection operators onto positive cones (see, e.g., \cite{HQ,Li24.2}). However, for isotone projection cones, the underlying orthogonal structure is no longer available, making the analysis substantially more delicate and mathematically interesting. Moreover, the applications of these variational properties to the stability analysis of complementarity problems have not yet been fully developed. These observations motivate the present study.

The main objective of this paper is to investigate differentiability properties and to derive explicit formulas for both the Fréchet coderivative and the Mordukhovich coderivative of the metric projection onto isotone projection cones. Our approach is based on a local representation of the projection mapping via an associated generating system in a neighborhood of a reference point. This representation enables us to reduce the coderivative analysis to tractable algebraic conditions and to obtain explicit expressions under suitable assumptions.

The obtained results are then applied to analyze the local behavior of the projection operator. In particular, we characterize the covering constant of the projection mapping, providing quantitative information on its Lipschitzian behavior. Furthermore, we establish verifiable conditions for the Aubin property of solution mappings to parametric nonlinear complementarity problems associated with isotone projection cones. These results contribute to the understanding of stability issues in complementarity systems and provide useful tools for sensitivity analysis in optimization.

The structure of the paper is as follows. Section 2 reviews the basic notions and tools from variational analysis that will be used throughout the paper, and recalls properties of lattice cones. Section 3 is devoted to coderivative computations for the projection operator onto isotone projection cones. Section 4 presents applications to the covering constant and the Aubin property of solution mappings. Finally, Section 5 concludes the paper.

	\section{Preliminaries}
	\setcounter{equation}{0}
We begin by introducing the notation used throughout the paper.
For a set $C\subset\mathbb{R}^n$, ${\rm int}\,C$ denotes its interior.
The notation $t\downarrow 0$ means that $t\to 0$ with $t>0$.
The set of natural numbers is denoted by $\mathbb{N}$, and
$\mathbb{N}^\ast:=\mathbb{N}\setminus\{0\}$.
For $x\in\mathbb{R}^n$ and $r>0$, $\mathbb{B}(x,r)$ denotes the closed ball
centered at $x$ with radius $r$, and we write $\mathbb{B}:=\mathbb{B}(0,1)$
for the unit ball.
	
	Below are basic notions and facts from variational analysis, which are frequently used in the following; see \cite{M18,RW98} for more details.
	
	\begin{Definition}[see \cite{M18,RW98}]
	{\rm	Let  $f:\R^n\to \R^m$ be a single-valued mapping.
		For $x,w\in \R^n$ with $w\neq 0$, if the limit
		\[
		\lim_{t\downarrow 0}\frac{f(x+tw)-f(x)}{t}
		=: f'(x;w)
		\]
		exists, then $f$ is said to be \emph{G\^ateaux directionally differentiable}
		at $x$ in the direction $w$.
		The vector $f'(x;w)$ is called the \emph{G\^ateaux directional derivative} of $f$
		at $x$ along $w$.
}
	\end{Definition}
	
	\begin{Definition}[see \cite{M18,RW98}]
	{\rm	Let $f:\R^n\to \R^m$ be a single-valued mapping.
		Given $\bar x\in \R^n$, if there exists a continuous linear operator
		$\nabla f(\bar x):\R^n\to \R^m$ such that
		\[
		\lim_{h\to 0}
		\frac{f(\bar x+h)-f(\bar x)-\nabla f(\bar x)h}{\|h\|}
		=0,
		\]
		then $f$ is said to be \emph{Fr\'echet differentiable} at $\bar x$, and
		$\nabla f(\bar x)$ is called the \emph{Fr\'echet derivative} of $f$ at $\bar x$.
		\medskip
		
		Moreover, if
		\[
		\lim_{(u,v)\to(\bar x,\bar x)}
		\frac{f(u)-f(v)-\nabla f(\bar x)(u-v)}{\|u-v\|}
		=0,
		\]
		then $f$ is said to be \emph{strictly Fr\'echet differentiable} at $\bar x$.
}	\end{Definition}

\begin{Definition}[see \cite{M18,RW98}]
{\rm	Let $\Omega\subset\mathbb{R}^n$ be nonempty and let $\bar x\in\Omega$.
	
	The \emph{ regular/Fréchet normal cone} to $\Omega$ at $\bar x$ is given by
	\[
	\widehat N_\Omega(\bar x)
	:=
	\Bigl\{
	z\in\mathbb{R}^n \ \Big|\ 
	\limsup_{u\xrightarrow{\Omega}\bar x}
	\frac{\langle z,u-\bar x\rangle}{\|u-\bar x\|}
	\le 0
	\Bigr\}.
	\]
	
	The \emph{limiting/Mordukhovich normal cone} to $\Omega$ at $\bar x$ is defined as
	the Painlevé--Kuratowski outer limit of regular normal cones, that is,
	\[
	N_\Omega(\bar x)
	:=
	\Limsup_{x\xrightarrow{\Omega}\bar x}\widehat N_\Omega(x)
	=
	\Bigl\{
	z\in\mathbb{R}^n \ \Big|\
	\exists\, x^k\to\bar x,\ z^k\to z
	\ \text{with}\ 
	z^k\in\widehat N_\Omega(x^k)\ \forall k
	\Bigr\}.
	\]
}\end{Definition}

If $\bar x\notin\Omega$, we set
\[
\widehat N_\Omega(\bar x)=N_\Omega(\bar x)=\emptyset
\]
by convention.  
When $\Omega$ is convex, the Fréchet and Mordukhovich normal cones coincide and
reduce to the classical normal cone of convex analysis.

\begin{Definition}[see \cite{M18,RW98}]
{\rm	Let $\Phi:\mathbb{R}^n \rightrightarrows \mathbb{R}^n$ be a set-valued mapping with graph
	\[
	\operatorname{gph}\Phi := \{(x,y)\in\mathbb{R}^n\times\mathbb{R}^n \mid y\in\Phi(x)\}.
	\]
	
	The set-valued mapping $\widehat D^\ast\Phi(\bar x,\bar y):\mathbb{R}^n\rightrightarrows\mathbb{R}^n$ defined by
	\[
	\widehat D^\ast\Phi(\bar x,\bar y)(w)
	:=
	\bigl\{
	v\in\mathbb{R}^n \ \big|\ (v,-w)\in \widehat N_{\operatorname{gph}\Phi}(\bar x,\bar y)
	\bigr\},
	\qquad w\in\mathbb{R}^n,
	\]
	is called the \emph{ regular/Fréchet coderivative} of $\Phi$ at $(\bar x,\bar y)$.
	
	Similarly, the multifunction $D^\ast\Phi(\bar x,\bar y):\mathbb{R}^n\rightrightarrows\mathbb{R}^n$ given by
	\[
	D^\ast\Phi(\bar x,\bar y)(w)
	:=
	\bigl\{
	v\in\mathbb{R}^n \ \big|\ (v,-w)\in N_{\operatorname{gph}\Phi}(\bar x,\bar y)
	\bigr\},
	\qquad w\in\mathbb{R}^n,
	\]
	is referred to as the \emph{limiting/Mordukhovich coderivative} of $\Phi$ at $(\bar x,\bar y)$.
}\end{Definition}
In the single-valued case $\Phi(\bar x)=\{\bar y\}$, we simply write 
$\widehat D^\ast\Phi(\bar x)$ and $D^\ast\Phi(\bar x)$ instead of 
$\widehat D^\ast\Phi(\bar x,\bar y)$ and $D^\ast\Phi(\bar x,\bar y)$, respectively. 
If $\Phi:\R^n\to \R^n$ is strictly Fréchet differentiable at $\bar x$, then
\[
D^\ast\Phi(\bar x)(y)
=
\widehat D^\ast\Phi(\bar x)(y)
=
\{\nabla\Phi(\bar x)^* y\},
\qquad \forall\, y\in \R^n.
\]

	The following result is given in \cite{H24(2)}, which provides a useful formula to calculate the Fréchet coderivatives of single-valued mappings on Hilbert spaces.
	\begin{Lemma} {\rm (see \cite{H24(2)})}\label{Lem1}
		Let  $f: \R^n\rightarrow \R^n$ be a Lipschitz  continuous mapping on $\R^n$. Then, the Fréchet coderivatives of $f$ at a point $\ox\in \R^n$ satisfies that, for any $y\in \R^n$,
		$$\widehat D^\ast f(\ox)(y)=\left\{z\in \R^n\big|\limsup\limits_{u\rightarrow \ox}\dfrac{\la z,u-\ox\ra-\la y,f(u)-f(\ox)\ra}{\|u-\ox\|} \leq 0 \right\}.$$
	\end{Lemma}
	
Let $n \in \mathbb{N}^*$ be fixed and set $N := \{1,\ldots,n\}$.  
For any $i,j \in N$, the notation $\delta_{i,j}$ denotes the Kronecker delta, that is,
\[
\delta_{i,j} =
\begin{cases}
	1, & \text{if } i=j,\\
	0, & \text{if } i\ne j.
\end{cases}
\]
\begin{Definition}
	{\rm (see \cite{NN10,R17,IN90}) Let $b_1, b_2, \ldots, b_n \in \mathbb{R}^n$ be linearly independent vectors.  
		The cone $K$ defined by
		\[
		K := \operatorname{cone}\{b_1, b_2, \ldots, b_n\}
		= \left\{
		x \in \mathbb{R}^n \ \middle|\ 
		x = \sum_{i=1}^n \alpha_i b_i,\ \alpha_i \ge 0 \text{ for } i\in N
		\right\}
		\]
		is called a \emph{latticial cone}.  
		In this case, we say that $K$ is generated by the vectors $b_1, b_2, \ldots, b_n$.  
		Every latticial cone is closed and convex.
}	\end{Definition}

For the reader’s convenience, we briefly recall some key properties of projections onto latticial cones, following \cite{M65,NN10,R17}.	

\begin{Lemma}[see \cite{NN10,R17}]
	Let $K \subset \mathbb{R}^n$ be a latticial cone generated by linearly independent vectors $b_1, b_2, \ldots, b_n$.  
	Then the polar cone of $K$ admits the representation
	\[
	K^\circ = \Big\{ \sum_{i=1}^n \mu_i u_i \ \Big|\ \mu_i \ge 0,\, i \in N \Big\},
	\]
	where for each $j \in N$, the vector $u_j$ is determined by the relations
	\[
	\langle u_j, b_i \rangle = -\delta_{ij}, \quad i \in N.
	\]
	Since $K$ is closed and convex, it follows that
	\[
	K = \{x \in \mathbb{R}^n \mid \langle y, x \rangle \le 0 \ \text{for all } y \in K^\circ\}.
	\]
	Consequently, by the above characterization,
	\begin{equation}\label{eq:K_condition}
		x \in K \quad \Longleftrightarrow \quad \langle x, u_i \rangle \le 0, \ \forall i \in N.
	\end{equation}
\end{Lemma}

\begin{Lemma} \label{lm28}[see \cite{NN10}]
	For any subset of indices $I \subset N$, the family of vectors 
	\[
	\{b_i \mid i \in I\} \cup \{u_j \mid j \in N \setminus I\}
	\]
	is linearly independent.
\end{Lemma}

\begin{Lemma}[see \cite{M65}, Moreau decomposition theorem]
	Let $K \subset \mathbb{R}^n$ be a closed convex cone and $x \in \mathbb{R}^n$.  
	The following statements are equivalent:
	\begin{enumerate}[(i)]
		\item $x = y + z$ with $y \in K$, $z \in K^\circ$, and $\langle y, z \rangle = 0$;
		\item $y = P_K(x)$ and $z = P_{K^\circ}(x)$.
	\end{enumerate}
\end{Lemma}

\begin{Lemma} {\rm(see \cite[Theorem~2]{NN10}, \cite[Theorem~4]{R17}) \label{lem1}}
Let $K=\operatorname{cone}\{b_1,\ldots,b_n\}$ be a lattice cone generated by linearly independent vectors $b_1,\ldots,b_n$. 
Let $\{u_j\}_{j\in N}$ be the generators of the polar cone $K^\circ$ determined by
\[
\langle u_j, b_i \rangle = -\delta_{i,j}, \qquad i,j\in N.
\]
Then, for every $x\in \mathbb{R}^n$ and any subset $I\subset N$, the vector $x$ admits a unique representation of the form
\begin{equation}\label{32}
	x = \sum_{i\in N\setminus I} \alpha_i b_i + \sum_{j\in I} \beta_j u_j.
\end{equation}
Moreover, there exists a unique subset $I\subset N$ such that the coefficients in \eqref{32} satisfy
\[
\beta_j > 0 \quad \text{for all } j\in I, 
\qquad 
\alpha_i \ge 0 \quad \text{for all } i\in N\setminus I.
\]
In this case, the metric projection of $x$ onto $K$ is given by
\[
P_K(x) = \sum_{i\in N\setminus I} \alpha_i b_i.
\]
\end{Lemma}

	\section{Coderivative of the Metric Projection onto the Isotone Projection Cones in $\R^n$}
In this section, we investigate projection operators onto isotone projection cones and recall the key structural properties of these cones, which will be essential for the coderivative analysis.
\begin{Definition}[see \cite{NN10,IN90,IN92}]
{\rm	A closed generating cone $K\subset\mathbb{R}^n$ is called an
	\emph{isotone projection cone} if for all $x,y\in\mathbb{R}^n$ with $y-x\in K$, one has
	\[
	P_K(y)-P_K(x)\in K.
	\]
	Equivalently, if $x\preceq y$ with respect to the partial order induced by $K$,
	then the metric projection $P_K$ preserves this order.
	
	Moreover, it is known from \cite[Theorem~8]{IN92} that a cone $K$ is an isotone
	projection cone in $\mathbb{R}^n$ if and only if it is a latticial cone generated
	by a system $\{b_i\}_{i\in N}$ with the associated vectors $\{u_j\}_{j\in N}$
	satisfying
	\[
	\langle u_j,u_k\rangle \le 0,
	\qquad \forall\, j\neq k.
	\]
}\end{Definition}
To further describe the structure of isotone projection cones, we recall some notions from matrix theory that will be used in the sequel. 
A real square matrix is called a \emph{Z-matrix} if all of its off-diagonal entries are nonpositive. 
An \emph{M-matrix} is a Z-matrix whose eigenvalues are positive. 
In particular, a symmetric matrix is an M-matrix if and only if it is a positive semidefinite Z-matrix.

As a first result, we provide a characteristic representation of elements in
$\mathbb{R}^n$ in terms of index sets associated with a given reference point.

\begin{Theorem} \label{Thm11}
		Let $K \subset \mathbb{R}^n$ be an isotone projection cone generated by linearly independent vectors 
	$b_1,\ldots,b_n$, and let $u_1,\ldots,u_n$ be the
	associated vectors satisfying $\langle b_i,u_j\rangle=-\delta_{ij}$.
	Assume that $I_1, I_2,$ and $I_3$ are pairwise disjoint subsets of $N$ such that
	$I_1\cup I_2\cup I_3 = N$.
	Then, for every $x\in\mathbb{R}^n$, there exist two disjoint subsets
	$A_x$ and $B_x$ of $I_2$ such that $x$ admits the representation
	\begin{equation} \label{ccc}
		x
		=
		\sum_{i\in I_1} \alpha_i b_i
		+
		\sum_{i\in A_x} \alpha_i b_i
		+
		\sum_{j\in B_x} \beta_j u_j
		+
		\sum_{j\in I_3} \beta_j u_j,
	\end{equation}
	with $\alpha_i>0$ for all $i\in A_x$ and $\beta_j>0$ for all $j\in B_x$.
\end{Theorem}
\begin{proof}
	If $I_2 = \emptyset$, then the system
	\(
	\{b_i : i\in I_1\} \cup \{u_j : j\in I_3\}
	\)
	forms a basis of $\mathbb{R}^n$; consequently, condition 
	\eqref{ccc} is automatically satisfied.
	
	Now suppose that $I_2\neq\emptyset$.  
	By Lemma \ref{lm28}, the system
	\[
	\{b_i : i\in I_1\cup I_2\} \cup \{u_j : j\in I_3\}
	\]
	is a basis of $\mathbb{R}^n$, and therefore every $x\in\mathbb{R}^n$ admits a representation of the form
	\begin{equation}\label{1a}
		x
		=
		\sum_{i\in I_1} \alpha_i^0\, b_i
		+
		\sum_{i\in I_2} \alpha_i^0\, b_i
		+
		\sum_{j\in I_3} \beta_j^0\, u_j.
	\end{equation}
	
If \(\alpha_i^0 \ge 0\) for all \(i \in I_2\), then we obtain the representation \eqref{ccc}. Otherwise, we set
	\[
	A_x^1 := \{i\in I_2 : \alpha_i^0\ge 0\}, 
	\qquad 
B_x^1 := \{i\in I_2 : \alpha_i^0< 0\}.
	\]
Observe that, by Lemma \ref{lm28}, the family
\(
\{b_i \mid i\in I_1\cup A_x^1\}\cup \{u_j \mid j\in I_3\cup B_x^1\}
\)
constitutes a basis. Hence, for each \(k\in B_x^1\), one can write
	\[
	-b_k
	=
	\sum_{i\in I_1\cup A_x^1} \alpha_{k,i}^0\, b_i
	\;+\;
	\sum_{j\in I_3\cup  B_x^1} \beta_{k,j}^0\, u_j.
	\]
	For $l\in I_3\cup  B_x^1$,
	\[
	\begin{array}{rl}
		\delta_{k,l} &= -\langle b_k, u_l \rangle \\[0.3em]
		&= \displaystyle \sum_{i\in I_1\cup  A_x^1} \alpha_{k,i}^0\, \langle b_i, u_l\rangle
		\;+\;
		\sum_{j\in I_3\cup  B_x^1} \beta_{k,j}^0\, \langle u_j, u_l\rangle \\[0.6em]
		&= \displaystyle 
		\sum_{j\in I_3\cup  B_x^1} \beta_{k,j}^0\, \langle u_j, u_l\rangle.
	\end{array}
	\]
	This shows that 
	\[
\beta^0_k := (\beta_{k,j}^0)^T_{j \in I_3 \cup  B_x^1}
	\]
	is a solution to the linear system
	\[
	G \beta_k^0 = e_k,
	\]
	where $G$ is the Gram matrix of the vectors $\{u_j\}_{j\in I_3 \cup  B_x^1}$ and 
	$e_k$ denotes the $k$-th canonical unit vector in $\mathbb{R}^{\,|I_3\cup B_x^1|}$.  
	Therefore,
	\[
	\beta^0_k = G^{-1} e_k.
	\]
	Since $G$ is an $M$-matrix, its inverse $G^{-1}$ is entrywise nonnegative. Consequently,
	\[
	\beta_{k,k}^0>0, \qquad \beta_{k,j}^0 \ge 0 \qquad \text{for all } j \in I_3\cup B_x^1.
	\]
	
	Substituting this expression into \eqref{1a} yields
	\begin{equation}\label{2}
		\begin{array}{rl}	x
			&=
			\displaystyle \sum_{i\in I_1} \alpha_i^0\, b_i
			\;+\;
			\sum_{i\in  A_x^1} \alpha_i^0\, b_i
			\;+\;
			\sum_{k\in B_x^1} \alpha_k^0\, b_k
			\;+\;
			\sum_{j\in  I_3} \beta_j^0\, u_j.\\
			&=\displaystyle\sum_{i\in I_1} \alpha_i^0\, b_i
			\;+\;
			\sum_{i\in   A_x^1} \alpha_i^0\, b_i
			\;+\;
			\sum_{k\in  B_x^1} -\alpha_k^0\, \left(\sum_{i\in I_1\cup  A_x^1} \alpha_{k,i}^0\, b_i
			\;+\;\sum_{j\in  I_3\cup  B_x^1} \beta_{k,j}^0\, u_j\right)
			\;+\;
			\sum_{j\in  I_3} \beta_j^0\, u_j\\
			&=\displaystyle\sum_{i\in I_1} \alpha_i^1\, b_i
			\;+\;
			\sum_{i\in A_x^1} \alpha_i^1\, b_i
			\;+\;
			\sum_{j\in B_x^1} \beta_j^1\, u_j
			\;+\;
			\sum_{j\in  I_3} \beta_j^1\, u_j,\\
		\end{array}
	\end{equation}
	where
	\[
	\begin{aligned}
		\beta_j^1 &:= \sum_{k\in   B_x^1} (-\alpha_k^0)\beta_{k,j}^0, 
		& j\in B_x^1,\\[0.4em]
		\beta_j^1 &:= \beta_j^0 + \sum_{k\in B_x^1} (-\alpha_k^0)\beta_{k,j}^0, 
		& j\in I_3,\\[0.4em]
		\alpha_i^1 &:= \alpha_i^0 + \sum_{k\in I_1\cup A_x^1} (-\alpha_k^0)\alpha_{k,i}^0, 
		& i\in I_1\cup  A_x^1.
	\end{aligned}
	\]
	Observe that, after this procedure, we obtain
	\[
	\beta_j^1>0 \quad \text{for all } j\in  B_x^1,
	\]
	and
	\[
	\beta_j^1\ge \beta_j^0 \quad \text{for all } j\in I_3.
	\]
	Moreover, since \(B_x^1\neq\emptyset\), it follows that \( A_x^1\) is a proper subset of \(I_2\).

If \(\alpha_i^1\geq 0\) for all \(i\in A_x^1\), then we arrive at the representation \eqref{ccc}. Otherwise, define
\[
A_x^2:=\{i\in A_x^1 : \alpha_i^1\ge 0\}, 
\qquad 
B_x^2:=\{i\in A_x^1 : \alpha_i^1<0\}.
\]
Repeating the above procedure, we obtain the representation
\[
x=
\sum_{i\in I_1} \alpha_i^2 b_i
+
\sum_{i\in A_x^2} \alpha_i^2 b_i
+
\sum_{j\in B_x^2} \beta_j^2 u_j
+
\sum_{j\in  B_x^1} \beta_j^2 u_j
+
\sum_{j\in I_3} 
\beta_j^2 u_j.
\]
In this representation, we have
\[
\beta_j^2 >  0
\quad \text{for all } j\in A_x^2,
\]
and
\[
\beta_j^2\ge \beta_j^1>0
\quad \text{for all } j\in  B_x^1,
\qquad
\beta_j^2\ge \beta_j^1
\quad \text{for all } j\in I_3.
\]
Continuing this procedure, we obtain a strictly decreasing sequence of sets
\[
I_2 \supsetneq A_x^1 \supsetneq A_x^2 \supsetneq \cdots .
\]
Since \(I_2\) is finite, the process terminates after finitely many steps, say after \(m\) steps with \(m<|I_2|\). Consequently, we arrive at the representation
\[
x=
\sum_{i\in I_1} \alpha_i^{m} b_i
+
\sum_{i\in A_x^m} \alpha_i^{m} b_i
+
\sum_{j\in B_x^m} \beta_j^{m} u_j
+\sum_{j\in B_x^{m-1}} \beta_j^{m} u_j+...+
\sum_{j\in  B_x^1} \beta_j^m u_j
+
\sum_{j\in I_3} 
\beta_j^m u_j
\]
which satisfies
\[
\alpha_i^m\ge 0
\quad \text{for all } i\in A_x^m,
\]
and
\[
\beta_j^m>0
\quad \text{for all } j\in B_x^1\cup\cdots\cup B_x^m,
\qquad
\beta_j^m\ge \beta_j^{m-1}\ge...\ge \beta_j^0
\quad \text{for all } j\in I_3.
\]
Setting
\[
A_x:=\{i\in A_x^m \mid \alpha_i^m>0\}
\quad \text{and} \quad
B_x:=B_x^1\cup\cdots\cup B_x^m,
\]
we obtain the representation \eqref{ccc}.

For every collection of index sets \(I_1, I_2, I_3\) satisfying the assumptions of the theorem and every \(x\in \mathbb{R}^n\), the above construction uniquely determines the sequence of sets \(\{A_x^i\}_{i=1}^m\) together with the corresponding coefficients. Therefore, the representation \eqref{ccc} is unique.

\end{proof}

\begin{Remark}
	\rm
	\begin{enumerate}
		\item[(i)]
		If $I_2 = N$, then the representation~\eqref{ccc} coincides with
		\eqref{32}. Consequently, this representation yields the explicit
		formula for the projection $P_K(x)$ in this case.
		
		\item[(ii)]
		If $I_1=\varnothing$, then representation~\eqref{ccc} reduces to the
		projection onto the cone
		\[
		K_2:=\operatorname{cone}\{b_i : i\in I_2\},
		\]
		namely,
		\[
		P_{K_2}(x)=\sum_{i\in A_x} \alpha_i b_i .
		\]
		
		\item[(iii)]
		Based on the proof of Theorem~\ref{Thm11}, we obtain the following algorithm
		for constructing the representation~\eqref{ccc}.
		
		\medskip
		\noindent
		\textbf{Step 1 (Initialization).}
		Set
		\[
		A_x^{0}:=I_2,\qquad B_x^{0}:=\varnothing.
		\]
		Represent $x$ with respect to the basis
		\(\{b_i : i\in I_1\cup A_x^{0}\}\cup\{u_j : j\in I_3\}\), that is,
		\[
		x
		=
		\sum_{i\in I_1} \alpha_i^{0} b_i
		+
		\sum_{i\in A_x^{0}} \alpha_i^{0} b_i
		+
		\sum_{j\in I_3} \beta_j^{0} u_j .
		\]
		
		\medskip
		\noindent
		\textbf{Step 2 (Update).}
		Assume that at iteration $k\ge 1$, $x$ admits the representation
		\[
		x
		=
		\sum_{i\in I_1} \alpha_i^{k} b_i
		+
		\sum_{i\in A_x^{k}} \alpha_i^{k} b_i
		+	\sum_{j\in B_x^{k}} \beta_j^{k} u_j +...+	\sum_{j\in  B_x^{1}} \beta_j^{k} u_j +
		\sum_{j\in I_3} \beta_j^{k} u_j .
		\]
		Define the updated index sets by
		\[
		A_x^{k+1}:=\{\, i\in A_x^{k} : \alpha_i^{k}\ge 0 \,\},\qquad
		B_x^{k+1}:=\{\, i\in A_x^{k} : \alpha_i^{k}<0 \,\}.
		\]
		Then represent \(x\) with respect to the basis
		\(\{b_i : i\in I_1\cup A_x^{k+1}\}\cup\{u_j : j\in I_3\cup B_x^{k+1}\}\) as
		\[
	x
	=
	\sum_{i\in I_1} \alpha_i^{k+1} b_i
	+
	\sum_{i\in A_x^{k+1}} \alpha_i^{k+1} b_i
	+	\sum_{j\in B_x^{k+1}} \beta_j^{k+1} u_j +...+	\sum_{j\in  B_x^{1}} \beta_j^{k+1} u_j +
	\sum_{j\in I_3} \beta_j^{k+1} u_j .
	\]
		
		\medskip
		\noindent
		\textbf{Stopping criterion.}
		If $A_x^{k+1}=A_x^{k}$, the algorithm terminates. Otherwise, set
		$k:=k+1$ and return to \textbf{Step~2}.
		
		\medskip
		\noindent
		Since the sequence $\{A_x^{k}\}_{k\ge 1}$ is decreasing with respect to
		set inclusion, the algorithm terminates after finitely many iterations.
		At termination, the resulting expression yields the desired
		representation~\eqref{ccc}.
		
		\item[(iv)]
		If $K$ is an arbitrary latticial cone, representation~\eqref{ccc}
		remains valid whenever at least one of the following conditions holds:
		\[
		I_1=\emptyset,\qquad I_3=\emptyset,\qquad \text{or}\quad |I_2|\le 1 .
		\]
	\end{enumerate}
\end{Remark}

For each $x\in\mathbb{R}^n$, by Lemma~\ref{lem1}, there exists a unique subset $I\subset N$ such that in the decomposition \eqref{32} the coefficients satisfy
\[
\beta_j>0 \ (j\in I), \qquad \alpha_i\ge 0 \ (i\in N\setminus I).
\]
Define the index sets
\[
I_x^+ := \{\, i\in  N\setminus I \mid \alpha_i>0 \,\}, \qquad
I_x^\bullet := \{\, i\in  N\setminus I \mid \alpha_i=0 \,\}, \qquad
I_x^- := I.
\].

By applying Theorem~\ref{Thm11}, we obtain the following result.

\begin{Lemma}\label{lem2}
	Let $K \subset \mathbb{R}^n$ be an isotone projection cone generated by
	linearly independent vectors $b_1,\ldots,b_n$, and let $\bar x\in\mathbb{R}^n$.
	Then there exists $\delta>0$ such that, for all $x\in B_\delta(\bar x)$,
	\[
	I_{\bar x}^+ \subset I_x^+,
	\qquad
	I_{\bar x}^- \subset I_x^-,
	\qquad
	I_x^\bullet \subset I_{\bar x}^\bullet .
	\]
\end{Lemma}

\begin{proof}
	By Theorem~\ref{Thm11}, applied with
	\(
	I_1=I^+_{\bar x}, 
	I_2=I^\bullet_{\bar x}, 
	I_3=I^-_{\bar x},
	\) for every \(x\in\mathbb{R}^n\), 
	there exist two disjoint subsets \(A_x\) and \(B_x\) of \(I^\bullet_{\bar x}\) such that \(x\) admits the representation
	\[
	x
	=
	\sum_{i\in I_{\bar x}^+} \alpha_i b_i
	+
	\sum_{i\in A_x} \alpha_i b_i
	+
	\sum_{j\in B_x} \beta_j u_j
	+
	\sum_{j\in I_{\bar x}^-} \beta_j u_j,
	\]
	where
	\(
	\alpha_i>0 \ \text{for all } i\in A_x,
	\
	\beta_j>0 \ \text{for all } j\in B_x .
	\)
	
	On the other hand, the reference point $\bar x$ admits the representation
	\[
	\bar x
	=
	\sum_{i\in I_{\bar x}^+} \bar\alpha_i b_i
	+
	\sum_{j\in I_{\bar x}^-} \bar\beta_j u_j,
	\]
	with
	\[
	\bar\alpha_i>0 \ \text{for all } i\in I_{\bar x}^+,
	\qquad
	\bar\beta_j>0 \ \text{for all } j\in I_{\bar x}^- .
	\]
	
	By continuity of the coefficients with respect to $x$, there exists
	$\delta>0$ such that for all $x\in B_\delta(\bar x)$ one has
	\[
	\alpha_i>0 \quad \text{for all } i\in I_{\bar x}^+,
	\qquad
	\beta_j>0 \quad \text{for all } j\in I_{\bar x}^- .
	\]
	Consequently,
	\[
	I_x^+
	=
	I_{\bar x}^+ \cup \{\, i\in A_x : \alpha_i>0 \,\},
	\qquad
	I_x^-
	=
	I_{\bar x}^- \cup B_x,
	\]
	which immediately implies
	\[
	I_{\bar x}^+ \subset I_x^+,
	\qquad
	I_{\bar x}^- \subset I_x^-,
	\qquad
	I_x^\bullet \subset I_{\bar x}^\bullet .
	\]
	This completes the proof.
\end{proof}

The following result describes the directional and Fr\'echet differentiability
properties of the metric projection onto an isotone projection cone.

\begin{Theorem}\label{df}
	Let $K\subset\mathbb{R}^n$ be an isotone projection cone generated by linearly
	independent vectors $b_1,\ldots,b_n$, and let $\bar x\in\mathbb{R}^n$.
	Then the metric projection $P_K:\mathbb{R}^n\to K$ enjoys the following
	differentiability properties.
	
	\medskip
	\noindent{\bf (i) The nondegenerate case $I^\bullet_{\bar x}=\emptyset$.}
	In this case, the mapping $P_K$ is strictly Fr\'echet differentiable at $\bar x$, and
	\[
	\nabla P_K(\bar x)(y)
	=
	B(B^{T}B)^{-1}B^{T}y,
	\qquad \forall\, y\in\mathbb{R}^n,
	\]
	where $B$ denotes the matrix whose columns are the vectors
	$\{\,b_i \mid i\in I_{\bar x}^+\,\}$.
	
	\medskip
\noindent{\bf (ii) The degenerate case \(I^\bullet_{\bar x}\neq\emptyset\).}
In this case, the projection mapping \(P_K\) is not Fr\'echet differentiable at
\(\bar x\), while it is directionally differentiable at \(\bar x\) in every direction
\(h\in\mathbb{R}^n\).

For any \(h\in\mathbb{R}^n\), Theorem~\ref{Thm11} yields two disjoint subsets
\(A_h^\bullet\) and \(B_h^\bullet\) of \(I^\bullet_{\bar x}\) 
such that
\[
h
=
\sum_{i\in I_{\bar x}^+} \alpha_i b_i
+
\sum_{i\in A_h^\bullet} \alpha_i b_i
+
\sum_{j\in B_h^\bullet} \beta_j u_j
+
\sum_{j\in I^-_{\bar x}} \beta_j u_j,
\]
where
\(
\alpha_i>0, \ \forall i\in A_h^\bullet,
\
\beta_j>0, \ \forall j\in B_h^\bullet.
\)\\
Then the directional derivative of \(P_K\) at \(\bar x\) in the direction \(h\) is given by
\[
P_K'(\bar x;h)
=
B(B^{T}B)^{-1}B^{T}h,
\]
where \(B\) is the matrix whose columns are the vectors
\[
\{\,b_i \mid i\in I_{\bar x}^+\cup A_h^\bullet\,\}.
\]
\end{Theorem}

\begin{proof}
	(i) When \(I^\bullet_{\bar x}=\emptyset\), it follows from Lemma~\ref{lem2} that there exists a neighborhood \(B_\delta(\bar x)\) such that each \(x\in B_\delta(\bar x)\) admits a representation of the form
	\[
	x
	=
	\sum_{i\in I_{\bar x}^+} \alpha_i\, b_i
	\;+\;
	\sum_{j\in I^-_{\bar x}} \beta_j\, u_j,
	\qquad
	\alpha_i>0\ (i\in I^+_{\bar x}),\ 
	\beta_j>0\ (j\in I^-_{\bar x}).
	\]
	It follows from Lemma~\ref{lem1} that
	\[
	P_K(x)=\sum_{i\in I_{\bar x}^+}\alpha_i\, b_i.
	\]
	Hence, in the neighborhood $B_\delta(\bar x)$, the metric projection $P_K$ coincides with the 
	orthogonal projection onto the subspace generated by $\{b_i : i\in I^+_{\bar x}\}$.  
	That is,
	\[
	P_K(x)=Lx \quad \text{for all } x\in B_\delta(\bar x),
	\]
	where 
	\[
	L := B(B^T B)^{-1}B^T,
	\]
	and $B$ is the matrix whose columns are the vectors $\{b_i : i\in I^+_{\bar x}\}$.  
	Therefore, $P_K$ is strictly Fréchet differentiable at $\bar x$, and
	\[
	\nabla P_K(\bar x)(y)
	=
	B(B^T B)^{-1}B^T y,
	\qquad \forall\, y\in \mathbb{R}^n.
	\]
	
	\medskip
	(ii) Now assume $I^\bullet_{\bar x}\neq\emptyset$.
	Using Lemma~\ref{lem1} and Theorem~\ref{Thm11}, the point $\bar x$ admits the unique representation
	\[
	\bar x
	=
	\sum_{i\in I_{\bar x}^+} \bar\alpha_i b_i
	\;+\;
	\sum_{j\in I^-_{\bar x}} \bar\beta_j u_j,
	\qquad
	\bar\alpha_i>0,\ \bar\beta_j>0.
	\]
	
	For any direction $h\in\mathbb{R}^n$, consider its decomposition
	\[
	h
	=
	\sum_{i\in I_{\bar x}^+} \alpha_i b_i
	\;+\;
	\sum_{i\in A_h^\bullet} \alpha_i b_i
	\;+\;
	\sum_{j\in B_h^\bullet} \beta_j u_j
	\;+\;
	\sum_{j\in I^-_{\bar x}} \beta_j u_j,
	\]
	where the index sets satisfy
	\[
	A_h^\bullet, B_h^\bullet \subset I_{\bar x}^\bullet,\qquad
	A_h^\bullet \cap B_h^\bullet = \emptyset,
	\]
	and
	\[
	\alpha_i>0\ (i\in A_h^\bullet), 
	\qquad 
	\beta_j>0\ (j\in B_h^\bullet).
	\]
	
	Then, for all sufficiently small \(t>0\),
	\[
	\bar x + th
	=
	\sum_{i\in I_{\bar x}^+} (\bar\alpha_i + t\alpha_i)\, b_i
	\;+\;
	\sum_{i\in A_h^\bullet} t\alpha_i\, b_i
	\;+\;
	\sum_{j\in B_h^\bullet} t\beta_j\, u_j
	\;+\;
	\sum_{j\in I^-_{\bar x}} (\bar\beta_j+t\beta_j)\, u_j,
	\]
	where all coefficients appearing in the above representation are strictly positive.
	Therefore, Lemma~\ref{lem1} implies
	\[
	P_K(\bar x+th)
	=
	\sum_{i\in I_{\bar x}^+} (\bar\alpha_i+t\alpha_i)\, b_i
	\;+\;
	\sum_{i\in A_h^\bullet} t\alpha_i\, b_i.
	\]
		Since 
	\[
	P_K(\bar x)
	=
	\sum_{i\in I_{\bar x}^+} \bar\alpha_i\, b_i,
	\]
	we obtain
	\[
	\begin{aligned}
		P_K'(\bar x;h)
		&=
		\lim_{t\downarrow 0}\frac{P_K(\bar x+th)-P_K(\bar x)}{t} \\[0.2cm]
		&=
		\sum_{i\in I_{\bar x}^+} \alpha_i b_i
		\;+\;
		\sum_{i\in A_h^\bullet} \alpha_i b_i.
	\end{aligned}
	\]
	
	Finally, this expression is exactly the orthogonal projection of $h$ onto 
	$\operatorname{span}\{b_i : i\in I_{\bar x}^+ \cup A_h^\bullet\}$, i.e.,
	\[
	P_K'(\bar x;h)
	=
	B_h(B_h^T B_h)^{-1}B_h^T h,
	\]
	where $B_h$ is the matrix whose columns are the vectors $\{b_i : i\in I_{\bar x}^+ \cup A_h^\bullet\}$.

Since \(I^\bullet_{\bar x}\neq\emptyset\), there exist directions \(h_1,h_2\in\mathbb{R}^n\) such that
\[
A^\bullet_{h_1}\neq A^\bullet_{h_2},
\]
and consequently,
\[
B_{h_1}(B_{h_1}^{T}B_{h_1})^{-1}B_{h_1}^{T}
\neq
B_{h_2}(B_{h_2}^{T}B_{h_2})^{-1}B_{h_2}^{T}.
\]
Hence, the directional derivative mapping depends on the direction \(h\), which implies that \(P_K\) is not Fr\'echet differentiable at \(\bar x\).

\end{proof}

	The following result provides the formula for calculating the Fréchet coderivatives of the metric projection onto the isotone projection cone in $\R^n$.

\begin{Theorem}\label{Thm1}
	Let $K \subset \mathbb{R}^n$ be an isotone projection cone generated by linearly 
	independent vectors $b_1,\dots,b_n$, and let $\bar x \in \mathbb{R}^n$. 
	Then the Fréchet coderivative of the metric projection $P_K$ at $\bar x$ satisfies
	\begin{equation}\label{kq}
		\widehat D^\ast P_K(\bar x)(y)
		=
		\left\{
		\begin{aligned}
			z \in \mathbb{R}^n \;\Big|\;& 
			\langle z, b_i\rangle = \langle y, b_i\rangle,
			&& \forall\,  i \in I^+_{\bar x}, \\[2mm]
			& \langle z, b_i\rangle \le \langle y, b_i\rangle,
			&& \forall\, i \in I^\bullet_{\bar x}, \\[2mm]
			& \langle z, u_j\rangle \le 0,
			&& \forall\, j \in I^\bullet_{\bar x}, \\[2mm]
			& \langle z, u_j\rangle = 0,
			&& \forall\, j \in I^-_{\bar x}
		\end{aligned}
		\right\},
		\qquad \forall\, y \in \mathbb{R}^n.
	\end{equation}
\end{Theorem}

\begin{proof}
By Lemma~\ref{Lem1},
	\begin{equation}\label{eq:limsup}
		z \in \widehat D^\ast P_K(\bar x)(y)
		\iff 
		\limsup_{x\to\bar x}
		\frac{\langle z, x - \bar x\rangle - \langle y, P_K(x) - P_K(\bar x)\rangle}
		{\|x - \bar x\|}
		\le 0.
	\end{equation}
	
	Moreover, Lemma~\ref{lem2} ensures that, for $x$ sufficiently close to $\bar x$ we have 
	$I_{\bar x}^+ \subset I_x^+$ and $I_{\bar x}^- \subset I_x^-$.  
	Set $d := (x-\bar x)/\|x-\bar x\|$.  
	Assume that $d$ admits the representation
	\begin{equation}\label{d}
		d
		=
		\sum_{i\in I_{\bar x}^+} \alpha_i b_i
		+
		\sum_{i\in A_d^\bullet} \alpha_i b_i
		+
		\sum_{j\in B_d^\bullet} \beta_j u_j
		+
		\sum_{j\in I_{\bar x}^-} \beta_j u_j,
	\end{equation}
	where $A_d^\bullet$ and $B_d^\bullet$ are disjoint subsets of $I^\bullet_{\bar x}$,  
	with $\alpha_i>0$ for $i\in A_d^\bullet$ and $\beta_j>0$ for $j\in B_d^\bullet$.
	
	Since $P_K(x)$ coincides, for all $x$ sufficiently close to $\bar x$,  
	with the projection $A_d$  onto the subspace generated by  
	$\{b_i : i \in I_{\bar x}^+ \cup A_d^\bullet\}$, we obtain
	\[
	P_K(x) - P_K(\bar x)
	= A_d(x) - A_d(\bar x)
	= A_d(x-\bar x).
	\]

Thus,
\[
\begin{aligned}
	\frac{\langle z, x - \bar x\rangle - \langle y, P_K(x) - P_K(\bar x)\rangle}
	{\|x - \bar x\|}
	&= \langle z, d\rangle - \langle y, A_d(d)\rangle \\
	&= \langle z, d\rangle - \langle A_d^{T}(y), d\rangle \\
	&= \langle z - A_d(y), d\rangle.
\end{aligned}
\]

We next derive a characterization of the Fréchet coderivative of \(P_K\) at \(\bar x\). In particular, we establish that
\[
z \in \widehat D^\ast P_K(\bar x)(y)
\quad\Longleftrightarrow\quad
\sup_{\|d\|= 1} \langle z - A_d(y), d\rangle \le 0,
\]
which is equivalent to the condition
\begin{equation}\label{q}
	\langle z - A_d(y), d\rangle \le 0,\qquad \forall\, d\in\mathbb{R}^n.
\end{equation}

Let us first assume that \(z \in \widehat D^\ast P_K(\bar x)(y)\). Then
\begin{equation}\label{rg}
\limsup_{x\to \bar x}
\frac{\langle z, x - \bar x\rangle - \langle y, P_K(x) - P_K(\bar x)\rangle}
{\|x - \bar x\|}
\le 0.
\end{equation}
Fix \(d\in\mathbb{R}^n\) with \(\|d\|=1\). Let \(t_n \downarrow 0\) and define
\[
x_n := \bar x + t_n d.
\]
Then \(x_n \to \bar x\) and \(d_n = d\). Hence,
\[
\frac{\langle z, x_n - \bar x\rangle - \langle y, P_K(x_n) - P_K(\bar x)\rangle}
{\|x_n - \bar x\|}
=
\langle z - A_d(y), d\rangle.
\]
By \eqref{rg}, we obtain
\[
\langle z - A_d(y), d\rangle \le 0.
\]
Since \(d\in\mathbb{R}^n\) is arbitrary with \(\|d\|=1\), it follows that
\[
\sup_{\|d\|=1} \langle z - A_d(y), d\rangle \le 0.
\]

Conversely, assume that
\[
\sup_{\|d\|=1} \langle z - A_d(y), d\rangle \le 0.
\]
We show that
\[
\limsup_{x\to \bar x}
\frac{\langle z, x - \bar x\rangle - \langle y, P_K(x) - P_K(\bar x)\rangle}
{\|x - \bar x\|}
\le 0.
\]
Let \(x_n \to \bar x\), \(x_n \neq \bar x\), and define
\[
t_n := \|x_n - \bar x\|, 
\qquad 
d_n := \frac{x_n - \bar x}{t_n}, \quad \|d_n\|=1.
\]
Set
\[
\alpha := \limsup_{n\to\infty}
\frac{\langle z, x_n - \bar x\rangle - \langle y, P_K(x_n) - P_K(\bar x)\rangle}
{\|x_n - \bar x\|}.
\]
Then there exists a subsequence (still denoted by \(n\)) such that
\[
\frac{\langle z, x_n - \bar x\rangle - \langle y, P_K(x_n) - P_K(\bar x)\rangle}
{\|x_n - \bar x\|}
\to \alpha.
\]

Since the family \(\{A_d \mid d\in\mathbb{R}^n\}\) is finite, passing to a subsequence if necessary, we may assume that there exists \(\bar d\in\mathbb{R}^n\) such that
\[
A_{d_n}=A_{\bar d}
\qquad \text{for all } n\in\mathbb{N}^\ast.
\]
Consequently,
\[
\frac{\langle z, x_n - \bar x\rangle - \langle y, P_K(x_n) - P_K(\bar x)\rangle}
{\|x_n - \bar x\|}
=
\langle z - A_{\bar d}(y), d_n\rangle.
\]
Moreover,
\[
\lim_{n\to\infty}
\langle z - A_{\bar d}(y), d_n\rangle
=
\alpha.
\]
Therefore, for every \(\varepsilon>0\), there exists \(\bar n\in\mathbb{N}^\ast\) such that
\[
\alpha-\varepsilon
<
\langle z - A_{\bar d}(y), d_{\bar n}\rangle
=
\langle z - A_{d_{\bar n}}(y), d_{\bar n}\rangle
\le 0.
\]
Since \(\varepsilon>0\) is arbitrary, it follows that \(\alpha\le 0\). Consequently,
\[
\limsup_{x\to \bar x}
\frac{\langle z, x - \bar x\rangle - \langle y, P_K(x) - P_K(\bar x)\rangle}
{\|x - \bar x\|}
\le 0,
\]
which implies
\[
z \in \widehat D^\ast P_K(\bar x)(y).
\]

We now prove the main assertion.

	\medskip
	\noindent\textbf{(``$\subset$'').}
	Assume that $z \in \widehat D^\ast P_K(\bar x)(y)$.  
	Then \eqref{q} holds for every $d\in\mathbb{R}^n$.  
	
For each \(i\in I_{\bar x}^+\), choosing \(d=\pm b_i\) in \eqref{q}, together with the identities
\(
A^\bullet_{b_i}=A^\bullet_{-b_i}=\emptyset
\)
and
\(
A_{b_i}=A_{-b_i},
\)
yields
	\[
	\langle z - A_{b_i}(y), b_i\rangle = 0.
	\]
	Since $\langle A_{b_i}(y), b_i\rangle = \langle y, A_{b_i}(b_i)\rangle= \langle y, b_i\rangle$, we obtain
	\begin{equation}\label{1-eng}
		\langle z, b_i\rangle = \langle y, b_i\rangle.
	\end{equation}
	
	Next, for any $i,j\in I_{\bar x}^\bullet$, taking $d^1=b_i$ and $d^2=u_j$ in \eqref{q} yields
	\[
	\langle z - A_{d^1}(y), b_i\rangle \le 0,
	\qquad
	\langle z - A_{d^2}(y), u_j\rangle \le 0.
	\]
	Since 
	$\langle A_{d^1}(y), b_i\rangle =\langle y, A_{d^1}(b_i)\rangle= \langle y, b_i\rangle$  
	and  
	$\langle A_{d^2}(y), u_j\rangle =\langle y, A_{d^2}(u_j)\rangle= 0$,  
	we deduce
	\begin{equation}\label{2-eng}
		\langle z, b_i\rangle \le \langle y, b_i\rangle,
		\qquad
		\langle z, u_j\rangle \le 0.
	\end{equation}
	
Similarly, for every \(j\in I_{\bar x}^-\), choosing \(d=\pm u_j\) in \eqref{q}, together with the identity
\(
A_{u_j}=A_{-u_j},
\)
yields
\[
\langle z - A_{u_j}(y), u_j\rangle = 0.
\]
Since
\(
\langle A_{u_j}(y),u_j\rangle
=
\langle y, A_{u_j}(u_j)\rangle
=0,
\)
it follows that
\begin{equation}\label{3-eng}
	\langle z, u_j\rangle = 0.
\end{equation}
	
	Collecting \eqref{1-eng}--\eqref{3-eng} yields
	\[
	\widehat D^\ast P_K(\bar x)(y) \subset 
	\left\{
	\begin{aligned}
		z \in \mathbb{R}^n \;\Big|\;& 
		\langle z, b_i\rangle = \langle y, b_i\rangle \quad\forall i\in I^+_{\bar x}, \\[2mm]
		& \langle z, b_i\rangle \le \langle y, b_i\rangle \quad\forall i\in I^\bullet_{\bar x}, \\[2mm]
		& \langle z, u_j\rangle \le 0 \quad\forall j\in I^\bullet_{\bar x}, \\[2mm]
		& \langle z, u_j\rangle = 0 \quad\forall j\in I^-_{\bar x}.
	\end{aligned}
	\right\}.
	\]
	
	\medskip
	\noindent\textbf{(``$\supset$'').}
	Conversely, suppose that $z$ satisfies all conditions on the right-hand side of \eqref{kq}.  
	Let $d$ be arbitrary with representation \eqref{d}.  
	Using the imposed inequalities for \(\langle z,b_i\rangle\) and \(\langle z,u_j\rangle\), we have
	\[
	\begin{aligned}
		\langle z - A_d(y), d\rangle
		&=
		\sum_{i\in I_{\bar x}^+} 
		\alpha_i\,\langle z - A_d(y), b_i\rangle
		+
		\sum_{i\in A_d^\bullet} 
		\alpha_i\,\langle z - A_d(y), b_i\rangle \\[1mm]
		&\quad+
		\sum_{j\in B_d^\bullet} 
		\beta_j\,\langle z - A_d(y), u_j\rangle
		+
		\sum_{j\in I_{\bar x}^-} 
		\beta_j\,\langle z - A_d(y), u_j\rangle \\[2mm]
		&=
		\sum_{i\in I_{\bar x}^+} 
		\alpha_i\big(\langle z,b_i\rangle - \langle A_d(y),b_i\rangle\big)
		+
		\sum_{i\in A_d^\bullet} 
		\alpha_i\,\langle z - y, b_i\rangle \\
		&\quad+
		\sum_{j\in B_d^\bullet} 
		\beta_j\,\langle z, u_j\rangle
		+
		\sum_{j\in I_{\bar x}^-} 
		\beta_j\big(\langle z,u_j\rangle - \langle A_d(y),u_j\rangle\big).
	\end{aligned}
	\]
	Now, observe that the first and the fourth sums vanish. Indeed, for every \(i\in I_{\bar x}^+\),
	\[
	\langle z - A_d(y), b_i\rangle
	=
	\langle z,b_i\rangle - \langle y, A_d(b_i)\rangle
	=
	\langle z,b_i\rangle - \langle y,b_i\rangle
	=0,
	\]
	since \(A_d(b_i)=b_i\). Similarly, for every \(j\in I_{\bar x}^-\),
	\[
	\langle z - A_d(y), u_j\rangle
	=
	\langle z,u_j\rangle - \langle y, A_d(u_j)\rangle
	=
	\langle z,u_j\rangle 
	=0.
	\]
	Hence,
	\[
	\begin{aligned}
		\langle z - A_d(y), d\rangle
		&=
		\sum_{i\in A_d^\bullet} 
		\alpha_i\,\langle z - y, b_i\rangle
		+
		\sum_{j\in B_d^\bullet} 
		\beta_j\,\langle z, u_j\rangle
		\le 0.
	\end{aligned}
	\]
	Thus $\langle z - A_d(y), d\rangle \le 0$ for all $d$, and by \eqref{q} we conclude that  
	\(z\in \widehat D^\ast P_K(\bar x)(y)\).
	
	The proof is complete.
\end{proof}

In the special case where $K$ is a generalized orthant, that is, a cone generated by an 
orthogonal basis $\{b_1,\dots,b_n\}$, we obtain the following result.

\begin{Corollary}\label{Co1}
	Let $K \subset \mathbb{R}^n$ be a generalized orthant generated by an orthogonal basis 
	$\{b_1,\dots,b_n\}$, and let $\bar x \in \mathbb{R}^n$. 
	Then the Fréchet coderivative of the metric projection $P_K$ at $\bar x$ is given by  
	\begin{equation}\label{kq2}
		\widehat D^\ast P_K(\bar x)(y)
		=
		\left\{
		z \in \mathbb{R}^n \;\Big|\;
		\begin{aligned}
			&\langle z, b_i\rangle = \langle y, b_i\rangle,
			&&  \forall\, i \in I^+_{\bar x}, \\[1mm]
			&0 \le \langle z, b_i\rangle \le \langle y, b_i\rangle,
			&& \forall\, i \in I^\bullet_{\bar x},\\[1mm]
			&\langle z, b_i\rangle = 0,
			&& \forall\, i \in I^-_{\bar x}
		\end{aligned}
		\right\},
		\qquad \forall\, y \in \mathbb{R}^n .
	\end{equation}
\end{Corollary}

\begin{proof}
	Since $\{b_i : i \in N\}$ is an orthogonal basis, it follows that 
	\[
	u_j = -\frac{b_j}{\|b_j\|^2}, \qquad \forall j \in N.
	\]
	Substituting this relation into \eqref{kq} immediately yields \eqref{kq2}.
\end{proof}

\begin{Remark}
{\rm	When $K$ is the positive cone in $\mathbb{R}^n$, formula \eqref{kq2} reduces exactly to the results established in \cite[Theorem~3.1]{HQ} and \cite[Theorem~4.1]{Li24.2}.
}\end{Remark}

Next,  we derive the limiting (Mordukhovich) coderivative $D^\ast P_{K}(\bar x)$.  
Since $P_K$ is continuous, its graph is closed; hence, by \cite[Equation~8(18)]{RW98}, one has
\begin{equation}\label{dn}
	D^\ast P_{K}(\bar x)(y)
	=
	\Limsup_{x \to \bar x,\;\; y' \to y}
	\widehat D^\ast P_{K}(x)(y').
\end{equation}
	
	This, together with  Theorem \ref{Thm1}, allows us to provide a complete characterization of $D^\ast P_{K}(\ox)$.

\begin{Theorem}\label{Thm34}
	Let $K\subset\mathbb{R}^n$ be an isotone projection cone generated by
	linearly independent vectors $b_1,\dots,b_n$, and suppose that
	\[
	\bar x
	= \sum_{i\in I^+_{\bar x}} \bar\alpha_i b_i
	+ \sum_{j\in I^-_{\bar x}} \bar\beta_j u_j .
	\]
	Then the Mordukhovich coderivative of $P_K$ at $\bar x$ is given by
\begin{equation}\label{kq22}
	D^\ast P_K(\bar x)(y)
	=
	\bigcup_{
			I, J\subset I^\bullet_{\bar x},\ I\cap J=\emptyset
	}
	\left\{
	\begin{aligned}
		z\in\mathbb{R}^n\ \Big|\ &
		\langle z,b_i\rangle=\langle y,b_i\rangle,
		&& \forall\, i\in I^+_{\bar x}\cup I, \\[2mm]
		&\langle z,b_i\rangle\le \langle y,b_i\rangle,
		&& \forall\, i\in I^\bullet_{\bar x}\setminus(I\cup J), \\[2mm]
		&\langle z,u_j\rangle=0,
		&& \forall\, j\in I^-_{\bar x}\cup J, \\[2mm]
		&\langle z,u_j\rangle\le 0,
		&& \forall\, j\in I^\bullet_{\bar x}\setminus(I\cup J)
	\end{aligned}
	\right\},
	\ \forall\, y\in\mathbb{R}^n .
\end{equation}
\end{Theorem}
\begin{proof}
	\textbf{(``\(\supset\)'' .)}
Assume that \(z \in \mathbb{R}^n\) satisfies all conditions on the right-hand side of \eqref{kq22}. Then there exist disjoint subsets \(I, J \subset I^\bullet_{\bar x}\) such that
	\begin{equation} \label{3.15}
	\begin{cases}
		\langle z, b_i\rangle = \langle y, b_i\rangle,
		& 
		\forall\, i \in I^+_{\bar x}\cup I, \\[0.3em]
		\langle z, b_i\rangle \le \langle y, b_i\rangle,
		&  \forall\, i \in I^\bullet_{\bar x}\setminus(I\cup J), \\[0.3em]
		\langle z, u_j\rangle \le 0,
		& \forall\, j \in I^\bullet_{\bar x}\setminus(I\cup J, \\[0.3em]
		\langle z, u_j\rangle = 0,
		& \forall\, j\in I^-_{\bar x}\cup J.
	\end{cases}
	\end{equation}
For each $k \in \mathbb{N^*}$, set $y^k := y$, $z^k := z$, and construct
	$x^k \in \mathbb{R}^n$ by
	\[
	x^k
	:=
	\sum_{i \in I^+_{\bar x}} \bar\alpha_i b_i
	+ \sum_{i \in I} \frac{1}{k} b_i
	+ \sum_{j \in J} \frac{1}{k} u_j
	+ \sum_{j \in I^-_{\bar x}} \bar\beta_j u_j .
	\]
	Clearly $x^k \to \bar x$ as $k \to \infty$, and 
	\[
	I^+_{x^k} = I^+_{\ox}\cup I, \qquad  
	I^\bullet_{x^k} = I^\bullet_{\ox}\setminus(I\cup J), \qquad  
	I^-_{x^k} = I^-_{\ox}\cup J,
	\quad \text{for all } k.
	\]
Combining this with \eqref{3.15}, we obtain
	\[
	\begin{cases}
		\langle z^k, b_i\rangle = \langle y^k, b_i\rangle,
		& 
		 \forall\, i \in I^+_{x^k}, \\[0.3em]
		\langle z^k, b_i\rangle \le \langle y^k, b_i\rangle,
		&  \forall\, i \in I^\bullet_{x^k}, \\[0.3em]
		\langle z^k, u_j\rangle \le 0,
		& \forall\, j \in I^\bullet_{x^k}, \\[0.3em]
		\langle z^k, u_j\rangle = 0,
		& \forall\, j \in I^-_{x^k}.
	\end{cases}
	\]
	It follows from \eqref{kq} that $z^k \in \widehat{D}^{\,*} P_K(x^k)(y^k)$ for every $k$.  
	Passing to the limit and applying the definition of the Mordukhovich coderivative, 
	we conclude that $z \in D^{*}P_K(\bar x)(y)$.
	
	\medskip
	\textbf{(``\(\subset\)''.)}
	Now assume that $z \in D^{*}P_K(\bar x)(y)$. By definition, there exist sequences
	\[
	x^k \to \bar x,\qquad y^k \to y,\qquad z^k \to z,
	\]
	with $z^k \in \widehat D^{\,*}P_K(x^k)(y^k)$ for all $k$.  
	Hence, using \eqref{kq}, we have for each $k$:
	\begin{equation}\label{32new}
				\begin{cases}
			\langle z^k, b_i\rangle = \langle y^k, b_i\rangle,
			& 
			\forall\, i \in I^+_{x^k}, \\[0.3em]
			\langle z^k, b_i\rangle \le \langle y^k, b_i\rangle,
			&  \forall\, i \in I^\bullet_{x^k}, \\[0.3em]
			\langle z^k, u_j\rangle \le 0,
			& \forall\, j \in I^\bullet_{x^k}, \\[0.3em]
			\langle z^k, u_j\rangle = 0,
			& \forall\, j \in I^-_{x^k}.
		\end{cases}
		\end{equation}
Since the collection of all subsets of \(N\) is finite, by passing to a subsequence if necessary, we may assume without loss of generality that there exist disjoint subsets \(I', J' \subset N\) such that
\[
I^+_{x^k}=I', \qquad I^-_{x^k}=J',
\qquad \text{and} \qquad
I^\bullet_{x^k}=N\setminus (I'\cup J')
\]
for all \(k\in\mathbb{N}^\ast\).
Note that \(x^k \to \bar x\), hence \(
I^+_{\bar x} \subset I' \ \text{and} \ I^-_{\bar x} \subset J'.
\)
Define
\[
I := I' \setminus I^+_{\bar x}, \qquad J := J' \setminus I^-_{\bar x}.
\]
Then we have
\[
I^+_{x^k}=I^+_{\bar x} \cup I,\ \
I^-_{x^k}=I^-_{\bar x} \cup J\
\ \text{and} \
I^\bullet_{x^k}=I^\bullet_{\bar x} \setminus (I\cup J),\ \operatorname{for\ all}\  k\in\mathbb{N}^\ast.
\]
	Taking limits in \eqref{32new} along this subsequence yields
	\[
	\begin{cases}
		\langle z, b_i\rangle = \langle y, b_i\rangle,
		&  \forall\, i \in I^+_{\bar x} \cup I, \\[0.3em]
		\langle z, b_i\rangle \le \langle y, b_i\rangle,
		&  \forall\, i \in I^\bullet_{\bar x} \setminus (I\cup J), \\[0.3em]
		\langle z, u_i\rangle \le 0,
		& \forall\, i \in I^\bullet_{\bar x} \setminus (I\cup J), \\[0.3em]
		\langle z, u_j\rangle = 0,
		& \forall\, j \in I^-_{\bar x} \cup J.
	\end{cases}
	\]
	These are exactly the conditions characterizing the set on the right–hand side of \eqref{kq22}.  
	Thus $z$ belongs to that set, and the proof is complete.
\end{proof}
	
Analogously to Corollary~\ref{Co1}, when $K$ is a generalized orthant we obtain the following
coderivative formula.

\begin{Corollary}\label{Co2}
	Let $K \subset \mathbb{R}^n$ be a generalized orthant generated by an orthogonal basis 
	$\{b_1,\dots,b_n\}$, and let $\bar x \in \mathbb{R}^n$. 
	Then the Mordukhovich coderivative of the metric projection $P_K$ at $\bar x$ is given by  
	\begin{equation}\label{kq3}
		\begin{aligned}
			D^\ast P_K(\bar x)(y)
			&=
		\bigcup_{
			I, J\subset I^\bullet_{\bar x},\ I\cap J=\emptyset
		}
		\left\{
		\begin{aligned}
			z\in\mathbb{R}^n\ \Big|\ &
			\langle z,b_i\rangle=\langle y,b_i\rangle,
			&& \forall\, i\in I^+_{\bar x}\cup I, \\[2mm]
			&0\le\langle z,b_i\rangle\le \langle y,b_i\rangle,
			&& \forall\, i\in I^\bullet_{\bar x}\setminus(I\cup J), \\[2mm]
			&\langle z,b_i\rangle=0,
			&& \forall\, i\in I^-_{\bar x}\cup J
		\end{aligned}
		\right\},\\[1mm]
			&\qquad\forall\, y \in \mathbb{R}^n .
		\end{aligned}
	\end{equation}
\end{Corollary}

\begin{proof}
	Proceeding in the same manner as in the proof of Corollary~\ref{Co1} and using Theorem~\ref{Thm34}, we obtain
\eqref{kq3}.
\end{proof}
	
\begin{Remark}
	{\rm 
		When $K$ is the positive cone in $\mathbb{R}^n$, formula~\eqref{kq3} coincides with the result of 
		\cite[Theorem~3.4]{HQ}.}
\end{Remark}

We now present an example illustrating the conclusion of
Theorems~\ref{Thm1} and~\ref{Thm34}.

\begin{Example}{\rm
		In $\mathbb{R}^2$, consider the nonnegative monotone cone
		\[
		K := \{x=(x_1,x_2)\in\mathbb{R}^2 \mid x_1 \ge x_2 \ge 0\}.
		\]
		This cone is an isotone projection cone and admits the latticial generators
		\[
		b_1=(1,0),\qquad b_2=(1,1),
		\]
		with the associated polar generators
		\[
		u_1=(-1,1),\qquad u_2=(0,-1).
		\]
		
	For every $z=(z_1,z_2)\in \mathbb{R}^2$, the following table lists the inner products $\langle z,b_i\rangle$ and $\langle z,u_j\rangle$:
	
\begin{figure}[h]
	\centering
	\includegraphics[width=0.7\textwidth,height=0.5\textheight,keepaspectratio]{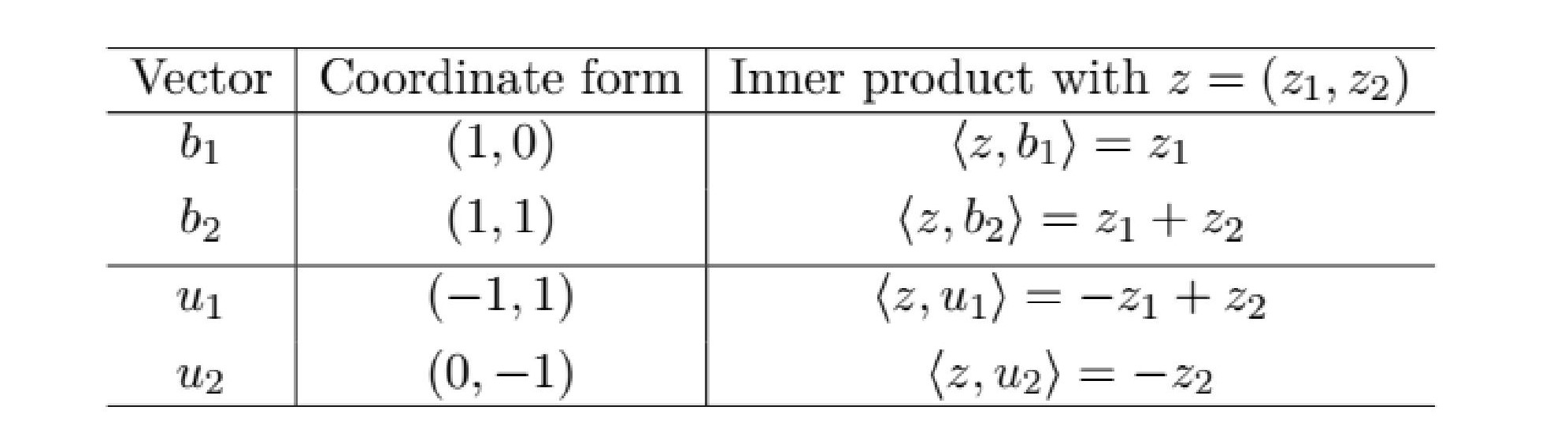}
\end{figure}

	\medskip
	Let $\bar x=(\bar x_1,\bar x_2)\in\mathbb{R}^2$. We analyze several cases.
		
		\medskip
		\noindent\textbf{Case 1.}
		If $\bar x_1>\bar x_2>0$, then $\bar x\in\operatorname{int}K$ and
		\[
		I^+_{\bar x}=\{1,2\},\qquad I^\bullet_{\bar x}=I^-_{\bar x}=\emptyset.
		\]
		Hence,
		\begin{equation}\label{1}
			\widehat D^*P_K(\bar x)(y)=D^*P_K(\bar x)(y)=\{y\}.
		\end{equation}
		
		\medskip
		\noindent\textbf{Case 2.}
		Assume $\bar x_1>0\ge \bar x_2$, so that $I^+_{\bar x}=\{1\}$.
		
		If $\bar x_2<0$, then $I^\bullet_{\bar x}=\emptyset$ and $I^-_{\bar x}=\{2\}$, and
		\begin{equation}\label{2}
			\widehat D^*P_K(\bar x)(y)=D^*P_K(\bar x)(y)=\{(y_1,0)\}.
		\end{equation}
		
		If $\bar x_2=0$, then $I^\bullet_{\bar x}=\{2\}$ and $I^-_{\bar x}=\emptyset$, and
		\[
		\widehat D^*P_K(\bar x)(y)
		=\{z=(y_1, z_2) \mid  0\le z_2\le y_2\},
		\]
		\begin{equation}\label{3}
			D^*P_K(\bar x)(y)
			=\{z = (y_1, z_2) \mid 0\le z_2\le y_2\}
			\cup\{y,(y_1,0)\}.
		\end{equation}
		
		\medskip
		\noindent\textbf{Case 3.}
		Suppose $-\bar x_2<\bar x_1\le \bar x_2$, yielding $I^+_{\bar x}=\{2\}$.
		
		If $\bar x_1<\bar x_2$, then $I^\bullet_{\bar x}=\emptyset$ and $I^-_{\bar x}=\{1\}$, and
		\begin{equation}\label{4}
			\widehat D^*P_K(\bar x)(y)=D^*P_K(\bar x)(y)
			=\left\{\left(\tfrac{y_1+y_2}{2},\tfrac{y_1+y_2}{2}\right)\right\}.
		\end{equation}
		
		If $\bar x_1=\bar x_2$, then $I^\bullet_{\bar x}=\{1\}$ and $I^-_{\bar x}=\emptyset$, and
		\[
		\widehat D^*P_K(\bar x)(y)
		=\left\{z=(z_1,y_1+y_2-z_1)\mid \tfrac{y_1+y_2}{2}\le z_1\le y_1\right\},
		\]
		\begin{equation}\label{5}
			\begin{aligned}
				D^*P_K(\bar x)(y)
				&=\left\{z=(z_1,y_1+y_2-z_1)\mid \tfrac{y_1+y_2}{2}\le z_1\le y_1\right\}\\
				&\quad\cup\left\{y,
				\left(\tfrac{y_1+y_2}{2},\tfrac{y_1+y_2}{2}\right)\right\}.
			\end{aligned}
		\end{equation}
		
		\medskip
		\noindent\textbf{Case 4.}
		Let $-\bar x_2\le \bar x_1\le 0$, so that $I^+_{\bar x}=\emptyset$.
		
		If $-\bar x_2<\bar x_1<0$, then $I^\bullet_{\bar x}=\emptyset$ and $I^-_{\bar x}=\{1,2\}$, and
		\begin{equation}\label{6}
			\widehat D^*P_K(\bar x)(y)=D^*P_K(\bar x)(y)=\{(0,0)\}.
		\end{equation}
		
		If $\bar x_1=0>-\bar x_2$, then $I^\bullet_{\bar x}=\{1\}$ and $I^-_{\bar x}=\{2\}$, and
		\[ \widehat D^*P_K(\bar x)(y)=\{(z_1,0)\mid 0\le z_1\le y_1\}, \]
		\begin{equation}\label{7}
			D^*P_K(\bar x)(y)
			=\{(z_1,0)\mid 0\le z_1\le y_1\}\cup\{(y_1,0),(0,0)\}.
		\end{equation}
		
		If $\bar x_1=-\bar x_2<0$, then $I^\bullet_{\bar x}=\{2\}$ and $I^-_{\bar x}=\{1\}$, and
			\[ \widehat D^*P_K(\bar x)(y)=\{(z_1,z_1)\mid 0\le z_1\le \tfrac{y_1+y_2}{2}\}, \]
		\begin{equation}\label{8}
			D^*P_K(\bar x)(y)
			=\{(z_1,z_1)\mid 0\le z_1\le \tfrac{y_1+y_2}{2}\}
			\cup\left\{(0,0),
			\left(\tfrac{y_1+y_2}{2},\tfrac{y_1+y_2}{2}\right)\right\}.
		\end{equation}
		
		If $\bar x_1=\bar x_2=0$, then $I^\bullet_{\bar x}=\{1,2\}$ and $I^-_{\bar x}=\emptyset$, and
		\begin{equation}\label{9}
			\widehat D^*P_K(0, 0)(y)
			=\{z=(z_1,z_2)\mid 0\le z_2\le z_1\le y_1,\ z_1+z_2\le y_1+y_2\}.
		\end{equation}
		
		Finally, the Mordukhovich coderivative $D^*P_K(0, 0)(y)$ is obtained as the union
		of the sets given in \eqref{1}--\eqref{9}.
	}
\end{Example}

	\section{Applications to Stability of Solution Mappings to Complementarity Problems and the Covering Constant}

In this section, we apply the coderivative formulas derived in the previous section. We begin by analyzing the covering constant of the projection operator and then derive verifiable conditions for the Aubin property of solution mappings to complementarity problems associated with isotone projection cones.

We begin by recalling the notion of the covering property and the associated covering constant for set-valued mappings. This property is equivalent to metric regularity, one of the fundamental concepts in variational analysis, and provides a quantitative tool for studying the local behavior and stability of mappings.
\begin{Definition} [see \cite{M18,RW98}]
	Let $\Phi:\mathbb{R}^n\rightrightarrows\mathbb{R}^m$ be a set-valued mapping, and let
	$U\subset\mathbb{R}^n$ and $V\subset\mathbb{R}^m$ be nonempty sets.
	We say that $\Phi$ has the \emph{covering property} on $U$ relative to $V$ if there exists
	$\kappa>0$ such that
	\begin{equation}\label{cp}
		\Phi(x)\cap V + \kappa r\,\mathbb{B}
		\subset
		\Phi(x+r\,\mathbb{B}),
		\qquad
		\text{whenever } x+r\,\mathbb{B}\subset U,\ r>0.
	\end{equation}
	
	The mapping $\Phi$ is said to have the \emph{$\kappa$-covering property} around
	$(\bar x,\bar y)\in\operatorname{gph}\Phi$ if there exist neighborhoods $U$ of $\bar x$
	and $V$ of $\bar y$ such that \eqref{cp} holds for all $x\in U$ and all $r>0$ with
	$x+r\,\mathbb{B}\subset U$.
	
	The supremum of all such $\kappa$ is called the \emph{exact covering bound} of $\Phi$
	at $(\bar x,\bar y)$ and is denoted by
	\[
	\operatorname{cov}\Phi(\bar x,\bar y).
	\]
	
	The \emph{covering constant} of $\Phi$ at $(\bar x,\bar y)$ is defined as
	\[
	\widehat{\alpha}(\Phi,\bar x,\bar y)
	:=
	\sup_{\eta>0}
	\inf\Bigl\{
	\|z^\ast\| \ \Big|\
	z^\ast\in \widehat D^\ast\Phi(x,y)(w^\ast),\ 
	x\in\mathbb{B}(\bar x,\eta),\
	y\in\Phi(x)\cap\mathbb{B}(\bar y,\eta),\
	\|w^\ast\|=1
	\Bigr\}.
	\]
\end{Definition}

We now apply the above notion to the metric projection onto isotone projection cones. Using the Fréchet coderivative formulas established in the previous section, we explicitly compute the covering constant of the projection operator.
\begin{Theorem}
	Let $\bar x\in\mathbb{R}^n$ and $\bar y=P_K(\bar x)$, where $K\subset\mathbb{R}^n$ is an isotone projection cone.
	Then the covering constant of the metric projection operator $P_K$ satisfies the following properties:
	\begin{enumerate}
		\item[(i)]
		$\widehat \alpha(P_K,\bar x,\bar y)=1$ for all $\bar x\in\operatorname{int}K$.
		\item[(ii)]
		$\widehat \alpha(P_K,\bar x,\bar y)=0$ for all $\bar x\notin\operatorname{int}K$.
	\end{enumerate}
\end{Theorem}
\begin{proof}
	\textbf{(i)}
	Let $\bar x\in\operatorname{int}K$.  
	By Theorem~\ref{df}, the projection mapping $P_K$ is strictly Fréchet differentiable at $\bar x$ with
	\[
	\nabla P_K(\bar x)=I_{\mathbb{R}^n}.
	\]
	Hence, by \cite[Theorems~1.57 and~4.1]{M18}, the regular covering constant of $P_K$ at $(\bar x,\bar y)$ satisfies
	\[
	\begin{aligned}
		\widehat{\alpha}(P_K,\bar x,\bar y)
		&=\operatorname{cov} P_K(\bar x,\bar y) \\
		&=\inf\bigl\{ \|\nabla P_K(\bar x)^* y^*\| \;\big|\; \|y^*\|=1 \bigr\} \\
		&=\inf\bigl\{ \|y^*\| \;\big|\; \|y^*\|=1 \bigr\}
		=1.
	\end{aligned}
	\]
	
	\medskip
	\textbf{(ii)}
	Assume that $\bar x\notin\operatorname{int}K$.  
	Then there exists an index $j\in N$ such that $j\notin I^+_{\bar x}$.  
	For any $\eta>0$, define
	\[
	x:=\bar x+\frac{\eta}{2\|u_j\|}\,u_j.
	\]
	Clearly, $x\in\mathbb{B}(\bar x,\eta)$ and $j\in I^-_x$.  
	Let
	\[
	y:=\frac{u_j}{\|u_j\|}.
	\]
Then,
\(\|y\|=1\ \operatorname{and }\
\langle y, b_i\rangle =0,
\ \forall\, i\neq j \in I^-_{\bar x}.
\)
By Theorem~\ref{Thm1}, we obtain
	\[
	0\in \widehat D^* P_K(x)(y).
	\]
Therefore,
\[
\inf\left\{
\|z\| \;\middle|\;
z\in \widehat D^*P_K(x)(y),\;
x\in\mathbb{B}(\bar x,\eta),\;
\|y\|=1
\right\}=0,
\qquad \forall\, \eta>0.
\]
Consequently,
\[
\widehat{\alpha}(P_K,\bar x,\bar y)
=
\sup_{\eta>0}
\inf\left\{
\|z\| \;\middle|\;
z\in \widehat D^*P_K(x)(y),\;
x\in\mathbb{B}(\bar x,\eta),\;
\|y\|=1
\right\}
=0.
\]
\end{proof}

We consider the parametric nonlinear complementarity problem associated with a closed convex cone:
\begin{equation}\label{NCPp}
	\text{(NCP$_p$)}\qquad
	x\in K,\quad F(p,x)\in -K^\circ,\quad \langle x,F(p,x)\rangle=0.
\end{equation}
Here, $K\subset\mathbb{R}^n$ is a closed convex cone, 
$F:\mathbb{R}^m\times\mathbb{R}^n\to\mathbb{R}^n$ is a continuously differentiable mapping,
$p\in\mathbb{R}^m$ denotes a parameter vector, and 
$K^\circ$ stands for the polar cone of $K$.

Problem \eqref{NCPp} can be equivalently reformulated as a fixed-point problem
associated with the metric projection onto the cone $K$. More precisely,
a vector $x\in\mathbb{R}^n$ solves \eqref{NCPp} if and only if it satisfies
\begin{equation}\label{FixedPoint}
	x = P_K\bigl(x - F(p,x)\bigr),
\end{equation}
where $P_K$ denotes the metric projection onto $K$, (see e.g. Proposition 2 in \cite{IN90}).	
	Denote by $S: \R^m\rightrightarrows\R^n$ the solution mapping associated with problem \eqref{NCPp},
	defined by
	\begin{equation}\label{SolMap}
		S(p):=\bigl\{x\in K \,\big|\, F(p,x)\in -K^\circ,\ \langle x,F(p,x)\rangle=0\bigr\}.
	\end{equation}
We have
\begin{equation} \label{gphS}
	\begin{array}{rl}
		\operatorname{gph} S
		&= \left\{(p,x)\in \mathbb{R}^m \times \mathbb{R}^n
		\;\middle|\;
		x = P_K\bigl(x - F(p,x)\bigr)
		\right\} \\[6pt]
		&= \left\{(p,x)\in \mathbb{R}^m \times \mathbb{R}^n
		\;\middle|\;
		G(p,x)\in \operatorname{gph} P_K
		\right\},
	\end{array}
\end{equation}
where the mapping \(G:\mathbb{R}^m \times \mathbb{R}^n \to \mathbb{R}^n \times \mathbb{R}^n\) is defined by
\[
G(p,x):=
\begin{pmatrix}
	x - F(p,x)\\
	x
\end{pmatrix}.
\]

The following result, derived from \eqref{gphS}, establishes a relationship between the Mordukhovich coderivative of the solution mapping $S$ and the projection operator $P_K.$
		\begin{Theorem}\label{thm:coderivativeS}
		Let $K\subset\mathbb{R}^n$ be an isotone projection cone in $\R^n$. Assume that the following qualification condition holds:
	\begin{equation} \label{dk45}
	\left.
	\begin{array}{l}
		-(F'_p(\bar p,\bar x))^{T} v^* = 0,\\[4pt]
		u^* = v^* - (F'_x(\bar p,\bar x))^{T} v^*,\\[4pt]
		v^* \in D^*P_K\big(\bar x - F(\bar p,\bar x)\big)(u^*)
	\end{array}
	\right\}
	\;\Longrightarrow\;
	u^* = 0.
\end{equation}
 Then for every $x^* \in \mathbb{R}^n$ we have
	\begin{equation} \label{Thm41}
	D^*S(\bar p,\bar x)(x^*)
	\subset
	\left\{
	-(F'_p(\bar p,\bar x))^{T} v^*
	\;\Big|\;
	\begin{array}{l}
		u^* = x^* + v^* - (F'_x(\bar p,\bar x))^{T} v^*,\\[4pt]
		v^* \in D^*P_K\big(\bar x - F(\bar p,\bar x)\big)(u^*)
	\end{array}
	\right\}.
	\end{equation}
	\end{Theorem}
\begin{proof}
Using the representation \eqref{gphS}, we apply \cite[Theorem 6.10]{M94} or \cite[Theorem 6.14]{RW98} to obtain \eqref{Thm41}.
	
\end{proof}

\begin{Remark} \label{R42}
If $F'_p(\bar{p}, \bar{x})$ is surjective, then the condition \eqref{dk45} is satisfied. 
Moreover, the Jacobian of $G$ at $(\bar p, \bar x)$ is given by
\[
\nabla G(\bar{p}, \bar{x})=
\begin{pmatrix}
	- F'_p(\bar p,\bar x) & I - F'_x(\bar p,\bar x)\\
	0 & I
\end{pmatrix},
\]
which has full row rank due to the surjectivity of $F'_p(\bar{p}, \bar{x})$. Hence, the inclusion in \eqref{Thm41} becomes an equality (see \cite[Exercise 10.7]{RW98} and \cite[Proposition 1.112]{M18}).
\end{Remark}

We next investigate the Aubin property of set-valued mappings. 
This property plays a fundamental role in variational analysis and sensitivity theory, 
as it characterizes a local Lipschitz-like behavior of set-valued mappings under perturbations. 
Moreover, it is closely related to several regularity notions, including metric regularity and coderivative criteria. 
For a detailed treatment of the Aubin property and related topics, we refer the reader to \cite{DR14} and the references therein.

Recall (see, e.g., \cite{DR14}) that a set-valued mapping 
$\Phi:\mathbb{R}^s \rightrightarrows \mathbb{R}^n$ is said to possess the Aubin 
property around $(\bar y,\bar x)\in \operatorname{gph}\Phi$ if there exist constants 
$\kappa>0$ and $r>0$ such that
\[
\Phi(y_1)\cap \B_r(\bar x) \subset \Phi(y_2) 
+ \kappa \|y_1 - y_2\|\, \B
\quad \text{for all } y_1,y_2 \in \B_r(\bar y).
\]
It is well known that this property is equivalent to the metric regularity of the 
inverse mapping $\Phi^{-1}$ (see \cite{DR14}). Moreover, it admits a characterization 
in terms of the Mordukhovich coderivative. We next present the corresponding result.
\begin{Lemma}[Mordukhovich criterion, see \cite{DR14,M18,RW98}] \label{Mo}
Let \(\Phi:\mathbb{R}^s \rightrightarrows \mathbb{R}^n\) and let \((\bar y,\bar x)\in \operatorname{gph}\Phi\). 
Suppose that \(\operatorname{gph}\Phi\) is locally closed at \((\bar y,\bar x)\). 
Then \(\Phi\) has the Aubin property around \((\bar y,\bar x)\) if and only if
\[
D^*\Phi(\bar y,\bar x)(0)=\{0\}.
\]
\end{Lemma}

Building upon the coderivative of the projection operator $P_K$  derived in the previous section, and combining it with  Theorem~\ref{thm:coderivativeS} and the Mordukhovich criterion (Lemma~\ref{Mo}), 
we obtain the following characterization of the Aubin property for the solution mapping $S$ at $(\bar p,\bar x)$.

\begin{Theorem} \label{Thm4.4}
	Let $K\subset \mathbb{R}^n$ be an isotone projection cone generated by a system 
	$\{b_i\}_{i=1}^n$ with the associated dual system $\{u_i\}_{i=1}^n$. 
	Assume that $F'_p(\bar p, \bar x)$ is surjective. 
	Then the solution mapping $S$ to problem \eqref{NCPp} has the Aubin property 
	around $(\bar p,\bar x)$ if and only if the following implication holds for all 
	$I,J \subset I^\bullet$ with $I\cap J=\emptyset$:
\begin{equation} \label{4.7}
	\left.
	\begin{aligned}
		\langle (F'_x(\bar p,\bar x))^{T} v^*, b_i\rangle &= 0 
		&\quad &\forall i\in I^+\cup I,\\[4pt]
		\langle (F'_x(\bar p,\bar x))^{T} v^*, b_i\rangle &\leq 0,\quad 
		\langle v^*, u_i\rangle \leq 0 
		&\quad &\forall i\in I^\bullet\setminus (I\cup J),\\[4pt]
		\langle v^*, u_i\rangle &= 0 
		&\quad &\forall i\in I^-\cup J
	\end{aligned}
	\right\}
	\;\Longrightarrow\;
	v^* = 0.
\end{equation}
	where 
	\[
	I^+ := I^+_{\bar x - F(\bar p,\bar x)},\quad 
	I^- := I^-_{\bar x - F(\bar p,\bar x)},\quad 
	I^\bullet := I^\bullet_{\bar x - F(\bar p,\bar x)}.
	\]
\end{Theorem}
\begin{proof}
	Since $F'_p(\bar p,\bar x)$ is surjective, it follows from Remark~\ref{R42} and Theorem~\ref{thm:coderivativeS} that
	\[
	D^*S(\bar p,\bar x)(0)
	=
	\left\{
	-(F'_p(\bar p,\bar x))^{T} v^*
	\;\Big|\;
	v^* \in D^*P_K\big(\bar x - F(\bar p,\bar x)\big)\big(v^* - (F'_x(\bar p,\bar x))^{T} v^*\big)
	\right\}.
	\]
	
	It follows directly that \(0 \in D^*S(\bar p,\bar x)(0)\) by choosing \(v^*=0\).
Conversely, for \(y \in \mathbb{R}^m\), one has
\(
y \in D^*S(\bar p,\bar x)(0)\) if and only if there exists
such that
	\[
	y^* = -(F'_p(\bar p,\bar x))^{T} v^*
	\]
	and
	\[
	v^* \in D^*P_K\big(\bar x - F(\bar p,\bar x)\big)\big(v^* - (F'_x(\bar p,\bar x))^{T} v^*\big).
	\]
	
According to Theorem~\ref{Thm34}, the latter condition holds if and only if there exist 
subsets $I,J \subset I^\bullet$ with $I\cap J = \emptyset$ such that
\[
\left\{
\begin{aligned}
	\langle (F'_x(\bar p,\bar x))^{T} v^*, b_i\rangle &= 0 
	&\quad &\forall i\in I^+\cup I,\\[4pt]
	\langle (F'_x(\bar p,\bar x))^{T} v^*, b_i\rangle &\leq 0,\quad 
	\langle v^*, u_i\rangle \leq 0 
	&\quad &\forall i\in I^\bullet\setminus (I\cup J),\\[4pt]
	\langle v^*, u_i\rangle &= 0 
	&\quad &\forall i\in I^-\cup J
\end{aligned}
\right.
\]

Therefore, by Lemma~\ref{Mo}, the solution mapping $S$ has the Aubin property 
around $(\bar p,\bar x)$ if and only if condition~\eqref{4.7} holds.

\end{proof}

Finally, based on the above analysis, we present an example illustrating Theorem~\ref{Thm4.4}.

\begin{Example}{\rm
	Let $K\subset \mathbb{R}^3$ be the cone isotone projection  generated by
	\[
	b_1=(1,0,0),\quad b_2=(1,1,0),\quad b_3=(1,1,1),
	\]
	with the associated dual system
	\[
	u_1=(-1,1,0),\quad u_2=(0,-1,1),\quad u_3=(0,0,-1).\] 
	Consider the mapping
	\[
	F(p,x)=(x_1,0,0)-p.
	\]
	Then $F'_p(\bar p,\bar x)=-I$ is surjective and
	\[
	F'_x(\bar p,\bar x)=
	\begin{pmatrix}
		1&0&0\\
		0&0&0\\
		0&0&0
	\end{pmatrix},
	\qquad
	(F'_x(\bar p,\bar x))^T v^*=(v_1,0,0).
	\]

\textbf{Case 1. $|I^-_{\bar x - F(\bar p,\bar x)}|\geq 2$.}
Then the Aubin property holds at $(\bar p,\bar x)$. 
All possible configurations are listed below.

\begin{figure}[h]
	\centering
	\includegraphics[width=0.9\textwidth,height=0.6\textheight,keepaspectratio]{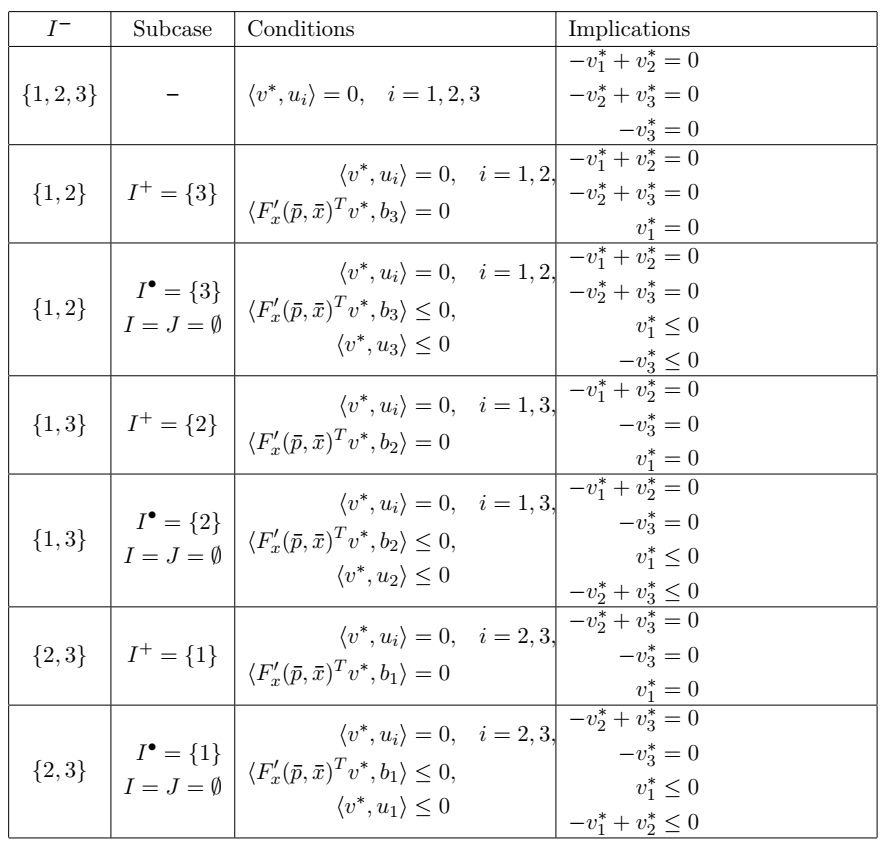}
\end{figure}
\FloatBarrier

\medskip
\noindent
In all cases, the above systems imply $v^*=0$. Hence, by Theorem \ref{Thm4.4}, the Aubin property holds at $(\bar p,\bar x)$.

\medskip
\noindent
\textbf{Case 2. $|I^-_{\bar x - F(\bar p,\bar x)}|\leq 1$.}
In this case, the Aubin property fails at $(\bar p,\bar x)$. 
All possible configurations are listed below.

\FloatBarrier
\begin{figure}[!htbp]
	\centering
\includegraphics[width=0.9\textwidth,height=0.6\textheight,keepaspectratio]{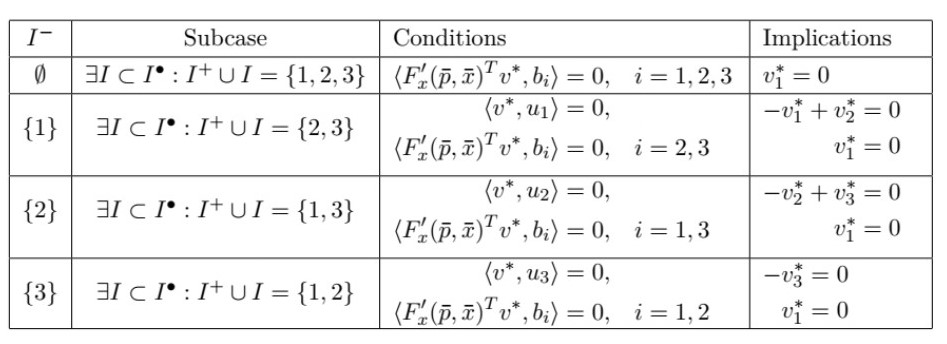}
\end{figure}

\medskip
\noindent
In these cases, the above systems do not imply $v^*=0$. Hence, by Theorem \ref{Thm4.4}, the Aubin property fails at $(\bar p,\bar x)$.

}
\end{Example}
	\section{Conclusions}
	
This paper has examined the local variational behavior of the metric projection onto isotone projection cones in $\mathbb{R}^n$. By employing a local representation of the projection mapping, we established explicit formulas for both the Fréchet and the Mordukhovich coderivatives. These results provide a transparent description of the structure of the projection operator and enable a precise analysis of its stability properties.

As an application, we characterized the covering constant of the projection mapping, revealing a clear distinction between interior and boundary regimes. Moreover, the obtained coderivative formulas were used to derive verifiable conditions ensuring the Aubin property for solution mappings to parametric complementarity problems associated with isotone projection cones.

A natural direction for future research is to extend the present analysis to general lattice cones. Such an extension is expected to provide further insights into the behavior of projection operators and, in particular, to yield useful information on projections onto polyhedral convex sets.
	\\
	
	
	{\bf Data Availability Statement} This manuscript has no associated data.
	
	{\bf Competing Interests} The author has no competing interests to declare that are relevant to the content of this article.
	

\end{document}